\documentclass[11pt]{article}
\usepackage{amsmath, amssymb, amscd, amsthm, amsfonts,mathrsfs,mathtools}
\usepackage{graphicx}
\usepackage{hyperref}
\usepackage{enumitem}
\usepackage[dvipsnames]{xcolor}

\usepackage[title]{appendix}

\title{Subordinators on  Feller Topological Monoids}
\author{
	Ulises Pérez Cendejas\footnote{UPC's research is supported by
CONACyT PhD scholarship CVU 882692.
} \and  Gerardo Pérez Suárez\footnote{
GPS's research is supported by
CONACyT PhD scholarship CVU 859066.
}}
\date{}

\newtheorem{theorem}{Theorem}
\newtheorem{lemma}[theorem]{Lemma}
\newtheorem{definition}[theorem]{Definition}

\newtheorem{proposition}[theorem]{Proposition}
\newtheorem{example}[theorem]{Example}
\newtheorem{remark}[theorem]{Remark}
\newtheorem{corollary}[theorem]{Corollary}

\begin{document}

\maketitle

\begin{abstract}
We investigate a class of topological monoids with a suitable family of characters which we call Feller topological monoids. 
We extend the classical notion of subordinators to subordinators on Feller topological monoids.
Under suitable assumptions, we prove a Lévy-Khintchine type representation for such subordinators.
In addition, a Lévy-Itô like decomposition is obtained.
These formulae generalize the classical ones for subordinators.
 
\end{abstract}



\section{Introduction}\label{section-introduction}

A subordinator, as defined in \cite[Chapter 3]{MR1406564}, is a stochastic process $(X_t)_{t\geq0}$ taking values in $[0,\infty)$ with independent and identical distributed increments. Subordinators satisfy the so-called \textit{Lévy-Khintchine formula}
\begin{align}\label{eqn:ClassicalLevyKhint}
    \mathbb{E}\left[e^{-\lambda X_t}\right]=\exp\left\{-t\left(\texttt{d}+\int_{(0,+\infty)}\left(1-e^{-\lambda x}\right)\pi(\text{d}x) \right)\right\},\quad t,\lambda\geq0,
\end{align}
where the constant $\texttt{d}\geq0$ and $\pi$ is a measure on $(0,+\infty)$ such that $\int_{(0,+\infty)}\min\{1,x\}\pi(\text{d}x)<\infty$. Another important feature of subordinators is that they possess a \textit{Lévy-Itô decomposition}, usually written as
\begin{align*}
    X_t=\texttt{d} t+\int_{(0,t]}\int_{(0+\infty)}x\mathcal{N}(\text{d}t,\text{d}x),\quad t\geq0,
\end{align*}
where $\mathcal{N}$ is a Poisson point process on $(0,+\infty)\times(0,+\infty)$ with intensity $\text{d}t\otimes\pi$. Writting $\mathcal{N}$ as $\mathcal{N}=\sum_{i}\delta_{\left(t_i,x_i\right)}$, the Lévy-Itô decomposition can be written as \begin{align}\label{eqn:ClassicalLevyIto}
    X_t=\texttt{d} t+\sum_{t_i\leq t}x_i,\quad t\geq0.
\end{align}
From now on, for reasons that will become clear later on, we will refer to these subordinators as \textit{additive subordinators}.

In the literature, one can find several kinds of stochastic processes which possess a Lévy-Khintchine type formula and a Lévy-Itô like decomposition. For instance,  the extremal processes \cite[Sec. 4.3]{MR2364939} and super-extremal processes \cite{MR1303927}, the random interlacements model \cite{MR2680403}, the flow of squared Bessel processes of dimension zero \cite[p. 508]{MR1725357} and \cite[Sec. 4]{MR656509}. Some of these Lévy-Khintchine formulae and Lévy-Itô decomposition can be found in Table \ref{tab:LevyKhintandItoInLit}.
\begin{table}[h!]
\begin{tabular}{ ||c| c c||}
\hline
  & Lévy-Khintchine type formula & Lévy-Itô like decomposition \\ [0.5ex] 
 \hline \hline
  \texttt{(a)}  &$\mathbb{E}[e^{-\lambda X_t}]=\exp\{-t\int_0^\infty\left(1-e^{-\lambda x}\right)\pi(\text{d}x)\}$ & $X_t=\sum_{t_i\leq t}x_i$ \\[5pt]
 \texttt{(b)}  &  $\mathbb{E}[e^{-\int_0^\infty \lambda(u)Z_x(u)\text{d}u }]=\exp\{-x\int_{\mathbf{W}_+}\left(1-e^{-\int_0^\infty \lambda(u)\omega(u) \text{d}u}\right)\eta(\text{d}\omega)\}$ &$Z_{\cdot}(x)=\sum_{y_i\leq x}\eta_i(\cdot)$ \\
 \texttt{(c)} & $\mathbb{E}[\boldsymbol{1}_{\{Y_t\leq\lambda\}}]=\exp\left\{-t\int_0^\infty \left(1-\boldsymbol{1}_{\{y\leq\lambda\}}\right)\pi(\text{d}y)\right\}$ & $Y_t=\bigvee_{t_i\leq t}y_i$ \\[5pt]  
 \texttt{(d)}  & $\mathbb{E}[\boldsymbol{1}_{\{\mathscr{J}_u\cap K\neq\emptyset\}}]=\exp\{-u\operatorname{cap}(K)\}$ & $\mathscr{J}_t=\bigcup_{t_i\leq t}\operatorname{range}(w_i^*)$    \\
 \hline
 
\end{tabular}
\caption{\label{tab:LevyKhintandItoInLit}The process $(X_t)_{t\geq0}$ is an additive subordinator with $\texttt{d}=0$.
The process $(Z_{\cdot}(x))_{x\geq0}$ is the flow of squared Bessel processes of dimension zero.
The process $(Y_t)_{t\geq0}$ is an extremal process. And the process $(\mathscr{J}_t)_{t\geq0}$ is the random interlacements model.}
\end{table}

Our aim in this paper is to investigate Lévy-Khintchine type formulae and Lévy-It\^o like decompositions for subordinators taking values in topological monoids.
These spaces provide a framework where it is possible to naturally  define these stochastic processes.
The examples shown in Table \ref{tab:LevyKhintandItoInLit} strongly motivated the
present work.

In order to work with topological monoids, we will assume that they possess a family of characters that characterize some of their convergence properties. 
We will call such family a determining class. In this direction, in a more or less equivalent approach,
Jonasson \cite{MR1625865} obtained general Lévy-Khintchine formulae for infinitely divisible random elements taking values in locally compact convex cones. One of the main differences between Jonasson's approach and ours is that while Jonasson works with continuous characters taking values in $(0,1)$, we drop these two assumptions. Moreover, we do not need the convex cone structure for the state space of our subordinators. Instead,  we only need to have a monoid with the aforementioned determining class. We call this topological space a Feller topological monoid.

In some other directions, the classical pathwise notion of Lévy processes, as defined in \cite{MR1406564}, has been extended  to more general groups, for instance, to Lie groups (see \cite{MR2060091}) or to some Banach spaces (see \cite{MR2353272}). In both cases the concept of stochastic integral with respect to a random measure is  as well generalized. In the same vein, subordinators have been extended to some particular cones of Banach spaces (see \cite{MR2264365} or \cite{MR2245374}).

The paper is organized as follows. In Section \ref{sec:TopoFellMon}, we introduce the definition of a Feller topological monoid and prove some of its properties. Then we gather some results about the Laplace transform on a topological monoid (as in \cite{MR2098271}). Afterwards we define subordinators and their Laplace exponent, and we generalize some well-known results related to this exponent. In Section \ref{sec:FellSemOnTopFellMon}, we use Feller semigroups techniques, namely Yosida approximation, to prove the Lévy-Khintchine formula. Finally, in Section  \ref{sec:LevyItoDec}, we obtain a Lévy-It\^o like decomposition for a class of subordinators.

\section{Topological Feller Monoids}\label{sec:TopoFellMon}

In this section we introduce the notion  of Feller topological monoid.
We also prove some of its properties.
Examples of Feller topological monoids that appear in the literature are presented as well.

First, let us recall some definitions. An abelian topological monoid is a triplet $(\mathbb{M},\tau,\oplus)$ such that $(\mathbb{M},\tau)$ is a topological space, $(\mathbb{M},\oplus)$ is a commutative monoid, and the operation $\oplus\colon\mathbb{M}\times\mathbb{M}\to\mathbb{M}$ is  continuous. 
We denote the neutral element in $(\mathbb{M},\oplus)$ by $\boldsymbol{e}$.
We refer the reader to \cite[Section 1.5]{MR2743117}  and the references therein for a complete exposition on topological monoids.

Let $\partial_\infty$ be an element not in $\mathbb{M}$ and set $\mathbb{M}^{\ast}\coloneqq\mathbb{M}\cup\{\partial_\infty\}$.
We denote by $(\mathbb{M}^{\ast},\tau^{\ast})$ the Alexandroff extension of $(\mathbb{M},\tau)$.
We recall that a sequence $(x_n)_{n\geq 1}\subset\mathbb{M}^{\ast}$ converges to $\partial_\infty$  if for any compact subset $K$ of $(\mathbb{M},\tau)$ there exists $N\geq 1$ such that $x_n\notin K$ for all $n\geq N$. 
As usual, we denote this  by $x_n\to\partial_\infty$.

We say that a function $\chi\colon\mathbb{M}\to[0,1]$ is a character on $\mathbb{M}$ if 
\begin{align*}
\chi(\boldsymbol{e}) = 1\quad  \textrm{and} \quad \chi(x \oplus y) = \chi(x)\chi(y)\quad \textrm{for all}\quad x,y\in\mathbb{M}.
\end{align*}
Equivalently, a character on $\mathbb{M}$ is a monoid homomorphism from $(\mathbb{M},\oplus)$ to $([0,1],\cdot)$, where $\cdot$ denotes the usual product on $\mathbb{R}$.
Let $\mathcal{B}(\mathbb{M})$ be the Borel $\sigma$-algebra of $(\mathbb{M},\tau)$.
We denote by $\textrm{Ch}(\mathbb{M})$ the set of all $\mathcal{B}(\mathbb{M})$-measurable characters on $\mathbb{M}$. 
We say that a subset $\widetilde{\mathbb{M}}$ of $\textrm{Ch}(\mathbb{M})$ is closed under pointwise multiplication if $\chi_1\chi_2\in\widetilde{\mathbb{M}}$ for any $\chi_1,\chi_2\in\widetilde{\mathbb{M}}$, where $(\chi_1\chi_2)(x)\coloneqq \chi_1(x)\chi_2(x)$, $x\in\mathbb{M}$.
The set of all characters in $\textrm{Ch}(\mathbb{M})$ that tend to zero at $\partial_\infty$ will be denoted by $\textrm{Ch}_0(\mathbb{M})$.

We are now ready to formulate our notion of Feller topological monoids.

\begin{definition}[Feller Topological Monoid\footnote{Terminological remark: While the term ``semigroup with identity'' seems to be more popular in the probabilistic literature when referring to a ``monoid'', we rather adopt the term “monoid” to avoid equivocal uses of the nowadays extended concept of Feller semigroups.}]\label{FC} Let $\mathbb{M}\equiv (\mathbb{M},\tau,\otimes)$ be an abelian topological monoid. 
We say that a pair $(\mathbb{M},\widetilde{\mathbb{M}})$ is a Feller topological monoid if $(\mathbb{M},\tau)$ is a locally compact second countable Hausdorff space,
and $\widetilde{\mathbb{M}}$ is a non-empty, countable and closed under pointwise multiplication  subset of $\emph{Ch}(\mathbb{M})$ such that for any sequence $(x_n)_{n\geq 1}\subset\mathbb{M}$ the following conditions are satisfied:  
\begin{enumerate}
    \item[(i)] $\chi(x_n)\to0$ for all $\chi\in\widetilde{\mathbb{M}}$ if and only if $x_n\to\partial_\infty$; 
    \item[(ii)] if $\chi(x_n)\to \psi(\chi)$ for every $\chi\in\widetilde{\mathbb{M}}$ and $\psi\not\equiv0$, then
    there exists $x\in\mathbb{M}$ such that $x_n\to x$ and  $\psi(\chi)=\chi(x)$ for each $\chi\in\widetilde{\mathbb{M}}$.
\end{enumerate}
In this case, we say that $\widetilde{\mathbb{M}}$ is a determining class of characters on ${\mathbb{M}}$.
\end{definition}

We present next some examples of Feller topological monoids that appear in the literature.

\begin{example}\label{ex:exp}
Let $\mathbb{R}_+\coloneqq[0,\infty)$ endowed with the usual topology $\tau_{\mathbb{R}_+}$ and the usual  sum $+$. Then $(\mathbb{R}_+,\tau_{\mathbb{R}_+},+)$ is an abelian topological monoid. 
For each $\lambda\in\mathbb{Q}^{+}\coloneqq \mathbb{Q}\cap (0,\infty)$, let $e_\lambda\colon\mathbb{R}_+\rightarrow (0,1]$ be the exponential function given by $e_\lambda(t)\coloneqq e^{-\lambda t}$ for each $t\in\mathbb{R}_+$. 
Set $\emph{\textrm{Exp}}(\mathbb{R}_+)\coloneqq\{e_\lambda:\lambda\in\mathbb{Q}^{+}\}$.
Notice that $\emph{\textrm{Exp}}(\mathbb{R}_+)$ is closed under pointwise multiplication.
Let $(t_n)_{n\geq 1}\subset\mathbb{R}_+$ and fix $\lambda\in\mathbb{Q}^{+}$. 
Then $(e_\lambda(t_n))_{n\geq 1}$ converges to zero if and only if $t_n\rightarrow\infty$. On the other hand,  $(e_\lambda(t_n))_{n\geq 1}$ converges to an element in $(0,\infty)$ if and only if  $(t_n)_{n\geq 1}$ converges in $\mathbb{R}_+$. Thus, $(\mathbb{R}_+,\emph{\textrm{Exp}}(\mathbb{R}_+))$ is a Feller topological monoid.
\end{example}

\begin{example}\label{ex:ind}
For each $x,y\in\mathbb{R}_+$, we set $x\vee y\coloneqq\max\{x,y\}$. 
Then $(\mathbb{R}_+,\tau_{\mathbb{R}_+},\vee)$ is an abelian topological monoid. 
For each $\lambda\in\mathbb{Q}^{+}$, we denote by $\emph{\textbf{1}}_{[0,\lambda]}$ the indicator function of the interval $[0,\lambda]$.
Set $\emph{\textrm{Ind}}(\mathbb{R}_+)\coloneqq \{\emph{\textbf{1}}_{[0,\lambda]}: \lambda\in\mathbb{Q}^{+}\}$.
Notice that $\emph{\textrm{Ind}}(\mathbb{R}_+)$ is closed under pointwise multiplication.
Let $(t_n)_{n\geq 1}\subset\mathbb{R}_+$.
Observe that $\emph{\textbf{1}}_{[0,\lambda]}(t_n)\rightarrow 0$ for all $\lambda\in\mathbb{Q}^{+}$ if and only if $t_n\rightarrow \infty$.
On the other hand, if $(t_n)_{n\geq 1}$ is bounded and $\emph{\textbf{1}}_{[0,\lambda]}(t_n)$ converges  to $\psi(\lambda)$  for all $\lambda\in\mathbb{Q}^{+}$, then  $(t_n)_{n\geq 1}$ converges to some $t$ and $\psi(\lambda)=\emph{\textbf{1}}_{[0,\lambda]}(t)$.
Moreover, $t_n\leq t$ for all $n$ sufficiently large.
Thus, $(\mathbb{R}_+,\emph{\textrm{Ind}}(\mathbb{R}_+))$ is a Feller topological monoid.
\end{example}

\begin{example}\label{ex:hit} For any integer $d\geq1$, let  $\mathbb{Z}_*^d$ be the space of proper subsets of $\mathbb{Z}^d$, equipped with the discrete topology $\tau_{\mathbb{Z}_*^{d}}$ and with the binary operation of the union $\cup$. Then $(\mathbb{Z}_*^{d},\tau_{\mathbb{Z}_*^d},\cup)$ is an abelian topological monoid. Here the neutral element $\boldsymbol{e}$ and the point at infinity $\partial_\infty$ are the empty set $\emptyset$ and the set $\mathbb{Z}^d$ respectively. Recall that the discrete topology on $\mathbb{Z}_*^d$ coincides with the topology induced by the set-theoretic convergence, that is to say, a sequence $(J_n)_{n\in\mathbb{N}}$ in $\mathbb{Z}_*$  converges to $J$ if and only if $J=\limsup_{n\to\infty}{J_n}=\liminf_{n\to\infty}J_n$ where, as usual,  $\limsup_{n\to\infty}{J_n}:=\cap _{n\geq 1}\cup _{k\geq n}J_{k}$ and $\liminf _{n\to \infty }J_{n}:=\cup _{n\geq 1}\cap _{k\geq n}J_{k}$. For a subset $K$ of $\mathbb{Z}^d$, we write $K\subset\subset\mathbb{Z}^d$, if $K$ has finite and non-zero cardinality. For each $K\subset\subset\mathbb{Z}^d$, we define the hitting functional $\chi_{K}:\mathbb{Z}_*^d\to\{0,1\}$ by $\chi_{K}(J):=\boldsymbol{1}_{\{J\cap K\neq\emptyset\}}$. Set $\operatorname{Hit}(\mathbb{Z}^d)=\{\chi_K:{K\subset\subset\mathbb{Z}^d}\}$. Notice that $\operatorname{Hit}(\mathbb{Z}^d)$ is closed under pointwise multiplication. Let $(J_n)_{n\geq1}\subset\mathbb{Z}_*^d$. Observe that $\chi_K(J_n)\to0$ for all $K\subset\subset\mathbb{Z}^d$ if and only if $\lim_{n\to\infty}J_n=\mathbb{Z}^d$. On the other hand,  $\limsup_{n\to\infty}J_n\neq\mathbb{Z}^d$ and $\chi_K(J_n)$ converges to $\psi(K)$ for all $K\subset\subset\mathbb{Z}^d$ if and only if $\lim_{n\to\infty}J_n=J$ for some $J\in\mathbb{Z}_*^d$ and $\psi(K)=\chi_K(J)$. Thus, $(\mathbb{Z}_*^d,\operatorname{Hit}(\mathbb{Z}^d))$ is a Feller topological monoid.
\end{example}   

In the following lemma we prove some useful properties of Feller topological monoid.

\begin{lemma}\label{lem:property}
Let $(\mathbb{M},\tau,\oplus,\widetilde{\mathbb{M}})$ be a Feller topological monoid.  
The following statements hold: 
\begin{enumerate}
    \item[(i)] $\widetilde{\mathbb{M}}$ is  non-vanishing everywhere, i.e., for any $x\in\mathbb{M}$ there exists a character $\chi\in\widetilde{\mathbb{M}}$ such that $\chi(x)\neq0.$
    \item[(ii)] $\widetilde{\mathbb{M}}$ separates points, i.e., for any two distinct  elements $x,y\in\mathbb{M}$ there exists $\chi\in\widetilde{\mathbb{M}}$ such that $\chi(x)\neq\chi(y)$.
    \item[(iii)] If $(x_n)_{n\geq 1}\subset \mathbb{M}$ is a sequence that converges to $\partial_{\infty}$, then $x_n\oplus y\to\partial_{\infty}$ for all $y\in\mathbb{M}$.
\end{enumerate}
\end{lemma}

\begin{proof}
Since $\mathbb{M}$ is locally compact, part (i) follows directly from Definition \ref{FC} (i). 
Let us now proceed to prove part (ii).
Let $x,y\in\mathbb{M}$, and suppose that $\chi(x)=\chi(y)$ for all $\chi\in\mathbb{M}$. Let $(z_n)_{n\geq 1}\subset\mathbb{M}$ be the sequence given by
\begin{align*}
    z_n\coloneqq 
    \begin{cases}
    x & \textrm{if } n \textrm{ is even},\\
    y & \textrm{if } n \textrm{ is odd}.
    \end{cases}
\end{align*}
Notice that $(\chi(z_n))_{n\geq 1}$ converges to $\chi(x)$ for all $\chi\in\widetilde{\mathbb{M}}$. 
On the other hand, from part (i) we know that there exists $\tilde{\chi}\in\widetilde{\mathbb{M}}$ such that $\tilde{\chi}(x)\neq 0$. 
We deduce from Definition \ref{FC} (ii) that $(z_n)_{n\geq 1}$ converges to some element in $\mathbb{M}$. Hence, $x=y$, i.e., $\widetilde{\mathbb{M}}$ separates points. 

It only remains to prove part (iii). Let $(x_n)_{n\geq 1}\subset\mathbb{M}$. If $x_n\to \partial_\infty$, then $\chi(x_n\oplus y)=\chi(x_n)\chi(y)\to 0$ for all $\chi\in\widetilde{\mathbb{M}}$. 
It follows from Definition \ref{FC} (i) that $x_n\oplus y\to \partial_\infty$. This concludes the proof.
\end{proof}

Our next task is to construct elements of the form $\bigoplus_{n\in\mathbb{N}}x_n$ for any sequence $(x_n)_{n\geq 1}$ in a Feller topological monoid $(\mathbb{M},\tau,\oplus,\widetilde{\mathbb{M}})$.

\begin{theorem}\label{proFellercone}
Let $(\mathbb{M},\tau,\oplus,\widetilde{\mathbb{M}})$ be a Feller topological monoid.
The following statements hold: 
\begin{enumerate}
    \item[(i)] (Divergence Criterion) If $(x_n)_{n\geq 1}\subset \mathbb{M}$ is a sequence such that $\prod_{n\geq 1}\chi(x_n)=0$ for all characters $\chi\in\widetilde{\mathbb{M}}$, then $\bigoplus_{i=1}^n x_i\to\partial_{\infty}$, as $n\to\infty$.
    \item[(ii)] (Convergence Criterion) If $(x_n)_{n\geq 1}\subset \mathbb{M}$ is a sequence such that $\prod_{n\geq 1}\chi(x_n)>0$ for some character $\chi\in\widetilde{\mathbb{M}}$, then there exists an element in $\mathbb{M}$ denoted by $\bigoplus_{i=1}^\infty x_i$
    such that \begin{align}
        \lim_{n\to\infty}\bigoplus_{i=1}^n x_i=\bigoplus_{i=1}^\infty x_i.
    \end{align}
    Furthermore, the definition of $\bigoplus_{i=1}^\infty x_i$ is independent of the  order of the $\oplus$-addends, so we can unambiguously write $\bigoplus_{n\in\mathbb{N}}x_n$.
    \item[(iii)] For any sequence $(x_n)_{n\geq1}\subset\mathbb{M}$,\begin{align*}
      \chi\left(\bigoplus_{i=1}^\infty x_i\right)=\prod_{i=1}^\infty\chi(x_i),\quad \chi\in\widetilde{\mathbb{M}},
    \end{align*}  
    where we set $\chi(\partial_\infty)=0$.
    \item[(iv)] For any partition $\{\mathcal{P},\mathcal{Q}\}$ of $\mathbb{N}$, we have $\left(\bigoplus_{n\in\mathcal{Q}}x_n\right)\oplus\left(\bigoplus_{n\in\mathcal{P}}x_n\right)=\left(\bigoplus_{n\in\mathbb{N}}x_n\right)$.
\end{enumerate}
\end{theorem}

\begin{proof}
Observe that the limit
$$
\prod_{n\geq 1}\chi(x_n)\coloneqq\lim_{m\rightarrow\infty}\prod_{n= 1}^{m}\chi(x_n)
$$
exists because the sequence $\big(\prod_{n= 1}^{m}\chi(x_n)\big)_{m\geq 1}$ is monotone decreasing for any $\chi\in\widetilde{\mathbb{M}}$.
Moreover, if $(\tilde{x}_{n})_{n\geq 1}$ is a reordering of $(x_{n})_{n\geq 1}$, then 
\begin{align}\label{reordering}
\prod_{n\geq 1}\chi(\tilde{x}_n)=\prod_{n\geq 1}\chi(x_n)\quad\textrm{for any}\quad \chi\in\widetilde{\mathbb{M}}.
\end{align}
For each $\chi\in\widetilde{\mathbb{M}}$, let 
$$\psi(\chi)\coloneqq \lim_{n\to\infty}\chi\left(\bigoplus_{i=1}^{n}x_i\right)=\lim_{m\rightarrow\infty}\prod_{n= 1}^{m}\chi(x_n).
$$
If $\psi(\chi)=0$ for all $\chi\in\widetilde{\mathbb{M}}$,
then it follows from
Definition \ref{FC} (i) that $\bigoplus_{i=1}^{n}x_i\to \partial_\infty$. 
This proves the divergence criterion in (i). 
On the other hand, if
$\psi\not\equiv 0$, 
then we deduce from Definition \ref{FC} (ii) that  $\left(\bigoplus_{i=1}^{n}x_i\right)_{n\geq 1}$ converges to some element in $\mathbb{M}$. We denote such element by $\bigoplus_{i=1}^{\infty}x_i$. 
Let $(\tilde{x}_{n})_{n\geq 1}$ be a reordering of $(x_{n})_{n\geq 1}$.
By using (\ref{reordering}) we obtain  that $\left(\bigoplus_{i=1}^{n}\tilde{x}_i\right)_{n\geq 1}$ converges to the element $\bigoplus_{i=1}^{\infty}\tilde{x}_i$. Moreover, from  Definition \ref{FC} (ii) and (\ref{reordering}) we get that $$\chi\left(\bigoplus_{i=1}^{\infty}x_i\right)=\chi\left(\bigoplus_{i=1}^{\infty}\tilde{x}_i\right)\quad\textrm{for all}\quad\chi\in\widetilde{\mathbb{M}}.$$
Hence, from Lemma \ref{lem:property} (ii) we conclude that $\bigoplus_{i=1}^{\infty}x_i=\bigoplus_{i=1}^{\infty}\tilde{x}_i$.
This proves the convergence criterion in (ii). Part (iii) follows readily from parts (i) and (ii).

Let $\mathcal{P}\subset \mathbb{N}$, and  let $(x_n)_{n\in\mathcal{P}}\subset\mathbb{M}$.
From the preceding parts, we know that 
$\bigoplus_{n=1}^{|\mathcal{P}|}\tilde{x}_n=\bigoplus_{n=1}^{|\mathcal{P}|}\hat{x}_n$ for any orderings $(\tilde{x}_n)_{n=1}^{|\mathcal{P}|}$ and $(\hat{x}_n)_{n=1}^{|\mathcal{P}|}$ of $(x_n)_{n\in\mathcal{P}}$.
Thus, we can define unambiguously the element $\bigoplus_{n\in\mathcal{P}}x_n\coloneqq\bigoplus_{n=1}^{|\mathcal{P}|}\tilde{x}_n$. 
Suppose now that $\{\mathcal{P},\mathcal{Q}\}$ is a partition of $\mathbb{N}$, and let $(x_n)_{n\geq 1}\subset\mathbb{M}$.
Let $(\tilde{x}_n)_{n=1}^{|\mathcal{P}|}$ and $(\hat{x}_n)_{n=1}^{|\mathcal{Q}|}$ be orderings of  $(x_n)_{n\in\mathcal{P}}$ and $(x_n)_{n\in\mathcal{Q}}$, respectively. 
By using Definition \ref{FC} and (\ref{reordering}) we obtain
\begin{align*}
    \chi\left(\left(\bigoplus_{n\in\mathcal{P}}x_n\right)\oplus\left(\bigoplus_{n\in\mathcal{Q}}x_n\right)\right)
    &=\chi\left(\bigoplus_{n\in\mathcal{P}}x_n\right)\chi\left(\bigoplus_{n\in\mathcal{Q}}x_n\right)\\
    &=\left(\prod_{n=1}^{|\mathcal{P}|}\chi(\tilde{x}_n)\right)\left(\prod_{n=1}^{|\mathcal{Q}|}\chi(\hat{x}_n)\right)\\
    &=\prod_{n=1}^{\infty}\chi(x_n)\\
    &= \chi\left(\bigoplus_{n\in\mathbb{N}}x_n\right)
\end{align*}
for all $\chi\in\widetilde{\mathbb{M}}$.
Therefore, it follows from Lemma \ref{lem:property} (ii) that $\left(\bigoplus_{n\in\mathcal{Q}}x_n\right)\oplus\left(\bigoplus_{n\in\mathcal{P}}x_n\right)=\left(\bigoplus_{n\in\mathbb{N}}x_n\right)$. This proves part (iv).

\end{proof}

Let $(\mathbb{M},\widetilde{\mathbb{M}})$ be a Feller topological monoid. We can define a natural metric $\rho_{\widetilde{\mathbb{M}}}$ on $\mathbb{M}$ given by 
\begin{align}\label{eqn:rhometric}
    \rho_{\widetilde{\mathbb{M}}}(x,y)=\sum_{n\geq1}\frac{|\chi_n(x)-\chi_n(y)|}{2^n},\quad x,y\in\mathbb{M},
\end{align}
where $(\chi_n)_{n\geq1}$ is some ordering of $\widetilde{\mathbb{M}}$.

Denote as $\tau_{\widetilde{\mathbb{M}}}$ the topology on $\mathbb{M}$ induced by the metric $\rho_{\widetilde{\mathbb{M}}}$, and $\widetilde{C}(\mathbb{M})$ the space of continuous functions from $\mathbb{M}$ into $\mathbb{R}$ with respect to the topology $\tau_{\widetilde{\mathbb{M}}}$. Setting $\chi_n(\partial_{\infty})=0$ for $n\geq1$, the formula (\ref{eqn:rhometric}) makes sense for any $x,y\in\mathbb{M}
^*$ and define a metric for $\mathbb{M}
^*$. By a slight abuse of
notation we also denote this metric as $\rho_{\widetilde{\mathbb{M}}}$. Let ${\mathcal{B}}(\tau_\mathbb{M})$ be the Borel $\sigma$-algebra of $(\mathbb{M},\tau_{\widetilde{\mathbb{M}}})$.

\begin{proposition}\label{pro:SigMet} The followings statements hold: 
\begin{enumerate}
    \item[(i)] If $\sigma\big(\widetilde{\mathbb{M}}\big)$ is the $\sigma$-algebra generated by the characters in $\widetilde{\mathbb{M}}$, then \begin{align*}
        \mathcal{B}(\mathbb{M})=\sigma(\widetilde{\mathbb{M}})=\mathcal{B}(\tau_{\widetilde{\mathbb{M}}}).
    \end{align*}    
    \item[(ii)] $(\mathbb{M},\rho_{\widetilde{\mathbb{M}}})$ is a  topological monoid, and $(\mathbb{M}^*,\rho_{\widetilde{\mathbb{M}}})$ is a compact space.
\end{enumerate}
\end{proposition}
\begin{proof}
We first proceed to prove (i). 
Notice that $\mathcal{B}(\tau_{\widetilde{\mathbb{M}}})\subset\sigma(\widetilde{\mathbb{M}})\subset\mathcal{B}(\mathbb{M})$.
On the other hand, since $(\mathbb{M},\tau)$ is in particular a Polish space and $\widetilde{\mathbb{M}}$ separates points, we can apply Lemma 3 in \cite{MR2673979} to obtain that the Borel $\sigma$-algebras of $(\mathbb{M}, \tau)$ and $(\mathbb{M},\rho_{\widetilde{\mathbb{M}}})$ coincide.
Hence, part (i) follows.

It only remains to prove the statement in (ii). Suppose $x_m\to x$, $y_m\to y$, as $m\to\infty$ , with respect to the metric $\rho_{\widetilde{\mathbb{M}}}$, this means that $\chi_n(x_m)\to\chi_n(x),\;\chi_n(y_m)\to\chi_n(y)$, as $m\to\infty$, for $n\geq1$. Then $\chi_n(x_m\oplus y_m)\to\chi(x\oplus y)$, as $m\to\infty$, for $n\geq1$, that is $x_m\oplus y_m\to x\oplus y$, as $m\to\infty$, with respect to the metric $\rho_{\widetilde{\mathbb{M}}}$, so that $(\mathbb{M},\rho_{\widetilde{\mathbb{M}}})$ is a topological monoid. 

Let us observe that $(\mathbb{M}^*,\rho_{\widetilde{\mathbb{M}}})$ is a compact space. Let $(x_m)_{m\in\mathbb{N}}$ be a sequence in $\mathbb{M}\cup\partial_\infty$ and consider the array $(\chi_n(x_m))_{n\geq1,m\geq1}$. By Cantor's diagonal method there exists an increasing sequence $(m(k))_{k\geq1}$ of integers such that $\lim_{k\to\infty} \chi_n(x_{m(k)})$ exists for $n\geq1$, by the definition of a Feller topological monoid, this means that there exists $x\in\mathbb{M}^*$ such that \begin{align*}
    \lim_{k\to\infty} \chi_n(x_{m(k)})=\chi_n(x),\quad n\geq1.
\end{align*}
\end{proof}

We now focus on a special class of Feller topological monoids. We say that a Feller topological monoid $(\mathbb{M},\tau,\oplus,\widetilde{\mathbb{M}})$ is idempotent if $(\mathbb{M},\oplus)$ is idempotent, i.e., if $x\oplus x=x$ for all $x\in\mathbb{M}$. Notice that the Feller topological moinoids $(\mathbb{R}_+,\operatorname{Ind}(\mathbb{R}_+))$ and $(\mathbb{Z}_+^d,\operatorname{Hit}(\mathbb{Z}^d))$ from Examples \ref{ex:ind} and \ref{ex:hit}, respectively, are idempotent.

\begin{lemma}\label{IdemCones} Let $(\mathbb{M},\widetilde{\mathbb{M}})$ be a Feller topological monoid.  The following is true:
\begin{enumerate} 
    \item[a)] $\mathbb{M}$ is  idempotent if and only if every character $\chi$ in $ \widetilde{\mathbb{M}}$ takes values on $\{0,1\}$.
    \item[b)] If $\mathbb{M}$ is idempotent, then the characters in $\widetilde{\mathbb{M}}$ are upper semicontinuous.
    \item[c)] If there exist $x\in\mathbb{M}\backslash \{\boldsymbol{e}\}$ and $\chi\in\widetilde{\mathbb{M}}$ satisfying $\chi(x)=1$, then $\mathbb{M}$ is not idempotent.
\end{enumerate}
\end{lemma}
\begin{proof}
To prove $a)$, first suppose that every character $\chi$ in $ \widetilde{\mathbb{M}}$ only takes values on $\{0,1\}$, then $\chi(x\oplus x)=\chi(x)^2=\chi(x)$ for all $\chi\in\widetilde{\mathbb{M}}$ and $x\in\mathbb{M}$, so that $x\oplus x=x$, because $\widetilde{\mathbb{M}}$ separates points. Conversely, assume that $x\oplus x=x$ for each $x\in\mathbb{M}$. So that, for every character $\chi$ in $\widetilde{\mathbb{M}}$, we have $\chi(x)^2=\chi(x)$, which implies that $\chi(x)\in\{0,1\}$ for all $x\in\mathbb{M}$.

For part $b)$, assume that $(x_n)_{n\in\mathbb{N}}$ is a sequence in $\mathbb{M}$ such that $x_n\to x$ as $n\to\infty$ and $\chi(x)=0$ for some $x\in\mathbb{M}$. For the sake of contradiction, suppose that  there exists a subsequence $(x_{n(k)})$ such that $\chi(x_{n(k)})=1$, so we can pass to a further subsequence such that $x_{m(k)}\to x'$, in $({\mathbb{M}},\rho_{\widetilde{\mathbb{M}}})$, with $x'\in\mathbb{M}^*$. By condition $(ii)$ from Definition \ref{FC}, we obtain $x=x'$, which contradicts that $\chi(x)=0$. So that \begin{align*}
    \limsup_n\chi(x_n)\leq\chi(x),\quad \chi\in\widetilde{\mathbb{M}},
\end{align*}
for every sequence $(x_n)_{n\in\mathbb{N}}$ such that $x_n\to x$.

To prove $c)$, note that part (ii) from Theorem \ref{proFellercone} ensures that $\lim_{n\to\infty}\oplus_{i=1}^nx=\oplus_{i=1}^\infty x\in\mathbb{M}$. We further have $\chi(\oplus_{i=1}^\infty x)=1$, and by part $(iv)$ from Theorem \ref{proFellercone}, $\left(\oplus_{i=1}^\infty x\right)\oplus \left(\oplus_{i=1}^\infty x\right)=\oplus_{i=1}^\infty x$. Since $x\neq\boldsymbol{e}$, there is a character $\chi_0\in\widetilde{\mathbb{M}}$ such that $\chi_0(x)\in[0,1)$ so that $\chi_0(\oplus_{i=1}^\infty x)=0$ which implies that $\oplus_{i=1}^\infty x\neq\boldsymbol{e}$. Then, $\mathbb{M}$ is not idempotent.
\end{proof}

For a determining class of characters $\widetilde{\mathbb{M}}$, we may consider the mapping $\varphi_{\widetilde{\mathbb{M}}}$ from ${\mathbb{M}}$ into $[0,1]$ defined as
\begin{align*}
   \varphi_{\widetilde{\mathbb{M}}}\equiv\sum_{n\geq1}\frac{1-\chi_n}{2^n},
\end{align*}
where $(\chi_n)_{n\geq1}$ is some ordering of $\widetilde{\mathbb{M}}$. The following condition will be useful, in order to obtain the Lévy-Khintchine formula for subordinators taking values in $\mathbb{M}$, in the forthcoming section.

\begin{enumerate}
    \item[\textbf{(I)}] The associated function $\varphi_{\widetilde{\mathbb{M}}}$ of a Feller topological monoid $(\mathbb{M},\widetilde{\mathbb{M}})$ satisfies
    \begin{align*}
    \alpha_{\widetilde{\mathbb{M}}}(\chi):=\lim_{x\to\boldsymbol{e}}\frac{1-\chi(x)}{\varphi_{\widetilde{\mathbb{M}}}(x)}\in[0,+\infty),\; \text{for all}\; \chi\in\widetilde{\mathbb{M}}.
\end{align*}
\end{enumerate}
Observe that, by part $(i)$ in Theorem \ref{proFellercone}, the previous ratio that appears in the limit is well-defined. 
Also, notice that condition (\textbf{I}) depends on the ordering $(\chi_n)_{n\geq 1}$ of $\widetilde{\mathbb{M}}$. For simplicity, we omit this dependence from our notation, which will be clear from the context.

\begin{example}
     Let $(\mathbb{M},\widetilde{\mathbb{M}})$ be an idempotent Feller topological monoid. 
     Then $(\mathbb{M},\widetilde{\mathbb{M}})$ satisfies condition \emph{\textbf{(I)}}. Even more, $\alpha_{\widetilde{\mathbb{M}}}\equiv0$ for any ordering of $\widetilde{\mathbb{M}}$; indeed, by part $b)$ in Lemma \ref{IdemCones}, there is $\varepsilon>0$ such that if $x\in B_\varepsilon(\boldsymbol{e})\backslash \{\boldsymbol{e}\}$ then $x\in\chi^{-1}(1)$ and $\varphi_{\widetilde{\mathbb{M}}}(x)>0$.
\end{example}

\begin{example} 
Let $(\mathbb{R}_+,\operatorname{Exp}(\mathbb{R}_+))$ be the Feller topological monoid in Example \ref{ex:exp}.
Suppose that $(\lambda_n)_{n\geq 1}$ is an ordering of $\mathbb{Q}^{+}$ such that $1/\alpha:=\sum_{n\geq1}\lambda_n/2^n<\infty$ (for example, Cantor's ordering $(\lambda_n)_{n\geq1}$ works since $\lambda_n\leq n,\;n\geq1$).
Then
$(\mathbb{R}_+,\operatorname{Exp}(\mathbb{R}_+))$ satisfies the condition \emph{\textbf{(I)}}
since
\begin{align*}
    \lim_{x\to0}\frac{1-e^{-\lambda x}}{\sum_{n\geq1}(1-e^{-\lambda_n x})/2^n}=\lambda\left(\sum_{n\geq1}\frac{\lambda_n}{2^n} \right)^{-1}=\lambda\alpha,\quad \lambda\in\mathbb{Q}_+,
\end{align*}
i.e., $\alpha_{\operatorname{Exp}(\mathbb{R}_+)}(e_\lambda)=\lambda\alpha$.
\end{example}

\subsection{Laplace Transform}
We gather here some of the fundamental properties of the Laplace transform on a Feller topological monoid $(\mathbb{M},\widetilde{\mathbb{M}})$, following nearly the approach in \cite[IX, \textsection5, No. 7]{MR2098271}.
For each finite measure $\mu$ on $(\mathbb{M},\mathcal{B}(\mathbb{M}))$, we define the Laplace transform of $\mu$ as the function $\mathscr{L}\mu$ from $\widetilde{\mathbb{M}}$ into $ [0,\infty)$ given by
\begin{align*}
    \mathscr{L}\mu(\chi)\coloneqq \int_{\mathbb{M}}\chi(s)\,\mathrm{d}\mu(s).
\end{align*}
Given a probability space $(\Omega,\mathcal{F},\mathbb{P})$ and an $\mathbb{M}$-valued random element $X$ on it, we define the Laplace transform of $X$ as the Laplace transform of the probability measure $\mathbb{P}(X\in\cdot)$, denoted as $\mathscr{L}_X$. Observe that $\mathscr{L}_X(\chi)=\mathbb{E}[\chi(X)]$ for every $\chi\in\widetilde{\mathbb{M}}$, where $\mathbb{E}$ is the expectation operator associated to $\mathbb{P}$.

A \textit{full submonoid} $\mathcal{S}$ is a subset of continuous characters, such that $\mathcal{S}\cap\textrm{Ch}_0(\mathbb{M})$ separates points, $\mathcal{S}\cap\textrm{Ch}_0(\mathbb{M})$ is non-vanishing everywhere, and $\chi_{\boldsymbol{1}}\in\mathcal{S}$, where $\chi_{\boldsymbol{1}}$ is the constant character equals to $1$. 

\begin{theorem}
Let $\mathcal{S}$ be a full submonoid of an abelian topological monoid $(\mathbb{M},\tau,\oplus)$. If $\mu_1$ and $\mu_2$ are two finite measure on $(\mathbb{M},\mathcal{B}(\mathbb{M}))$, such that $\mathscr{L}\mu_1$ and $\mathscr{L}\mu_2$ have the same restriction to $\mathcal{S}\cap\emph{\textrm{Ch}}_0(\mathbb{M})$, then $\mu_1=\mu_2$ on $(\mathbb{M},\mathcal{B}(\mathbb{M}))$.
\end{theorem}

\begin{corollary}Let $(\mathbb{M},\widetilde{\mathbb{M}})$ be a Feller topological monoid. If $\mu_1$ and $\mu_2$ are two finite measures on $(\mathbb{M},\mathcal{B}(\mathbb{M}))$ such that $\mathscr{L}\mu_1$ and $\mathscr{L}\mu_2$ have the same restriction to $\widetilde{\mathbb{M}}$, then $\mu_1=\mu_2$ on $(\mathbb{M},\mathcal{B}(\mathbb{M}))$.
\end{corollary}
\begin{proof}
By part $(i)$ in Proposition \ref{pro:SigMet}, $\mathcal{B}(\mathbb{M})=\mathcal{B}(\tau_{\widetilde{\mathbb{M}}})$. Therefore, $\mu_1$ and $\mu_2$ are two finite measure on $(\mathbb{M},\mathcal{B}(\tau_{\widetilde{\mathbb{M}}}))$. By part $(ii)$ from Proposition  \ref{pro:SigMet}, $(\mathbb{M},\rho_{\widetilde{\mathbb{M}}})$ is a topological monoid. Moreover, on $(\mathbb{M},\rho_{\widetilde{\mathbb{M}}})$ the set of characters $\mathcal{S}=\widetilde{\mathbb{M}}\cup\{\chi_{\boldsymbol{1}}\}$ is a full submonoid; indeed, by parts $(i)$ and $(ii)$ from Lemma \ref{lem:property},  $\mathcal{S}\cap \text{Ch}_0(\mathbb{M})=\widetilde{\mathbb{M}}$  separates points and is non-vanishing everywhere, and clearly the characters in $\widetilde{\mathbb{M}}$ are continuous in $(\mathbb{M},\rho_{\widetilde{\mathbb{M}}})$. The result follows immediately
from the previous theorem.
\end{proof}

\begin{corollary}Let $\mu_1$ and $\mu_2$ be two measures on $(\mathbb{M}\backslash\{\boldsymbol{e}\},\mathcal{B}(\mathbb{M}\backslash\{\boldsymbol{e}\}))$
such that
\begin{align}\label{eqn:protoLapExp}
\int_{\mathbb{M}\backslash\{\boldsymbol{e}\}} \big(1-\chi(x)\big) \mu_1(\text{d}x)=\int_{\mathbb{M}\backslash\{\boldsymbol{e}\}} \big(1-\chi(x)\big) \mu_2(\text{d}x)<+\infty,\quad \chi \in\widetilde{\mathbb{M}},
\end{align}
then $\mu_1=\mu_2$ on $(\mathbb{M}\backslash\{\boldsymbol{e}\},\mathcal{B}(\mathbb{M}\backslash\{\boldsymbol{e}\}))$.
\end{corollary}
\begin{proof}
The arguments are adapted from those on Theorem 1.23 in \cite{MR2760602}. Let us extend $\mu_1$ and $\mu_2$ by setting $\mu_1(\{\boldsymbol{e}\})=\mu_2(\{\boldsymbol{e}\})=0$. For fixed $\chi' \in\widetilde{\mathbb{M}}$, consider the measures $\nu_i$ for $i=1,2$, defined by $\nu_i(\text{d}x)=\big(1-\chi'(x)\big)\mu_i(\text{d}x)$ on $(\mathbb{M},\mathcal{B}(\mathbb{M}))$. Since $\widetilde{\mathbb{M}}$ is closed under pointwise multiplication, then
\begin{align*}
\int_{\mathbb{M}} \big(1-(\chi\cdot\chi')\big)(x) \mu_1(\text{d}x)=\int_{\mathbb{M}} \big(1-(\chi\cdot\chi')\big)(x) \mu_2(\text{d}x)<+\infty,\quad \chi \in\widetilde{\mathbb{M}}.
\end{align*}
By subtracting (\ref{eqn:protoLapExp}) from this last equation, we get that $\mathscr{L}\nu_1$ and $\mathscr{L}\nu_2$ coincides on $\mathbb{M}\cup\{\chi_{\boldsymbol{1}}\}$ and by the preceding corollary $\nu_1=\nu_2$ on $(\mathbb{M},\mathcal{B}(\mathbb{M}))$, for every $\chi'\in\widetilde{\mathbb{M}}$. Therefore, $\mu_1=\mu_2$ on
\begin{align*}
\bigcup_{\chi'\in\widetilde{\mathbb{M}}}\left\{x\in\mathbb{M}:1-\chi'(x)>0\right\}=\mathbb{M}\backslash\{\boldsymbol{e}\},
\end{align*}
where the last equality is due to the fact that $\widetilde{\mathbb{M}}$ is non-vanishing everywhere.
\end{proof}

\begin{lemma}\label{lem:weakconv} Let $\mu,(\mu_{n})_{n\geq1}$ be probability measures on $(\mathbb{M},\mathcal{B}(\mathbb{M}))$. Suposse that $\mathscr{L}_{\mu_n}(\chi)\to\mathscr{L}_{\mu}(\chi)$, for every $\chi\in\widetilde{\mathbb{M}}$, as $n$ goes to infinity. Then $\mu_n$ converges weakly to $\mu$ on $(\mathbb{M},\tau_{\widetilde{\mathbb{M}}})$ and $(\mathbb{M},\tau)$, as $n$ goes to infinity,
\end{lemma}
\begin{proof}
This is a straightforward application of Theorem 6 in \cite{MR2673979}.
\end{proof}

\subsection{Subordinators on Feller Topological Monoids}
Throughout this section we fix a Feller topological monoid $(\mathbb{M},\tau,\oplus,\widetilde{\mathbb{M}})$ and a probability space $(\Omega,\mathscr{F},\mathbb{P})$. The stochastic processes below will be $\mathbb{M}$-valued and defined on $(\Omega,\mathscr{F},\mathbb{P})$.

For now, we give below a definition of subordinators, which is enough for our immediate purposes. A more natural definition, from the perspective of Feller semigroups, will be given in the forthcoming section.
\begin{definition}[Subordinator]\label{PDS} We say that  $(X_t)_{t\geq0}$ is a subordinator if \begin{enumerate}
\item $X_0=\boldsymbol{e}$.
\item $(X_t)_{t\geq0}$ is right-continuous.
\item Given $t\geq0$, for every $s\geq0$, there exists a random element $X^{(t)}_s$, which could be defined on an enlargement of $(\Omega,\mathscr{F},\mathbb{P})$, such that $X_{s+t}\overset{d}{=}X^{(t)}_s\oplus X_t$ where $X_s^{(t)}\overset{d}{=}X_s$ and $X^{(t)}_s$ is independent of $X_t$. 
\end{enumerate}
\end{definition}

We are interested in the Laplace transform of these subordinators, i.e. \begin{align*}
    \mathscr{L}_{X_t}(\chi)=\mathbb{E}[\chi(X_t)],\quad \chi\in\widetilde{\mathbb{M}},\; t\geq0.
\end{align*}
From $3$ in Definition \ref{PDS} and since $\chi(x)\in[0,1]$, $x\in\mathbb{M}$, it follows that $t\mapsto\mathscr{L}_{X_t}(\chi)$ is a decreasing function which satisfies the Cauchy equation\begin{align*}
    \mathscr{L}_{X_{s+t}}(\chi)=\mathscr{L}_{X_s}(\chi)\cdot\mathscr{L}_{X_t}(\chi),\quad s,t\geq0.
\end{align*}
Therefore, there exists a mapping $\Psi$ from $\widetilde{\mathbb{M}}$ into $[0,+\infty]$ such that  \begin{align*}
 \mathscr{L}_{X_t}(\chi)=\exp{\{-t\Psi(\chi)\}},\quad \chi\in\widetilde{\mathbb{M}},\; t\geq0,
\end{align*}
which gives rise to the \textit{Laplace exponent} $\Psi$ \textit{of} $(X_t)_{t\geq0}$. 

For any integer $n\geq1$, let $\big(X_t^{(1)}\big)_{t\geq0},\cdots,\big(X_t^{(n)}\big)_{t\geq0}$ be independent subordinators having the same law as $(X_t)_{t\geq0}$, which could be defined on an enlargement of $(\Omega,\mathscr{F},\mathbb{P})$. We then have for $0\leq t_1<\cdots <t_n$,
\begin{align}
    \mathbb{E}[\chi(X_{t_1})\cdots\chi(X_{t_n})]&=\mathbb{E}\left[\chi(X_{t_1}^{(1)})\cdot\chi\left(X_{t_1}^{(1)}\oplus X_{t_2-t_1}^{(2)}\right)\cdots\chi\left(X_{t_1}^{(1)}\oplus\cdots \oplus X_{t_n-t_{n-1}}^{(n)}\right)\right]\nonumber\\
    &=e^{-t_1\Psi(\chi^n)}e^{-(t_2-t_1)\Psi(\chi^{n-1})}\cdots e^{-(t_n-t_{n-1})\Psi(\chi)}\label{eqn:FDD},
\end{align}
where $\chi^n$ is the product of $\chi$ with itself $n$-times.
\begin{proposition}
If $\Psi(\chi)>0$ for any $\chi\in\widetilde{\mathbb{M}}$, then $\lim_{t\to+\infty}X_t=\partial_\infty$ a.e.
\end{proposition}
\begin{proof}
Let $\chi\in\widetilde{\mathbb{M}}$ be fixed. The stochastic process  $(M_t(\chi))_{t\geq0}$, defined by $M_t(\chi)=\chi(X_t)\exp\{t\Psi(\chi)\}$, $t\geq0$, is a martingale with respect to the filtration $(\mathscr{F}_t)_{t\geq0}$ where $\mathscr{F}_t=\sigma(\chi(X_s):s\leq t)$. Indeed it is enough, by the monotone class
theorem, to show that for any times $s_1<\cdots <s_k\leq s<t$
\begin{align*}
    \mathbb{E}\left[M_t(\chi)\cdot\prod_{i=0}^k \chi(X_{s_i})\right]=\mathbb{E}\left[M_s(\chi)\cdot\prod_{i=0}^k \chi(X_{s_i})\right],
\end{align*}
but this equality follows readily by applying (\ref{eqn:FDD}) to both sides. Therefore, $(M_t(\chi))_{t\geq0}$ is a positive martingale for any $\chi\in\widetilde{\mathbb{M}}$ which implies $\lim_{t\to+\infty}\chi(X_t)\exp\{t\Psi(\chi)\}<+\infty$ for any $\chi\in\widetilde{\mathbb{M}}$ a.e. Since $\Psi(\chi)>0$, we get that $\lim_{t\to+\infty}\chi(X_t)=0$ for every $\chi\in\widetilde{\mathbb{M}}$ a.e. whence the result follows from (i) in Definition \ref{FC}.
\end{proof}
\subsubsection{Character Functionals of a Subordinator}
In this subsection, we will be interested in \textit{character functionals} of the form
\begin{align*}
   I_t=\int_0^t \chi(X_s)\text{d}s,\quad \chi\in\widetilde{\mathbb{M}},\;t\geq0,
\end{align*}
which are a generalization of the classical \textit{exponential functionals of Lévy processes} (see \cite{MR2178044}). The following proposition is an extension of Theorem 2 (i) in \cite{MR2178044}. However, the second part of the proof presented there does not work in this general setting. But a natural proof arises when this result is stated in our context, as we shall see next.
\begin{proposition}
Let $\boldsymbol{e}_q$ denote a random time which has exponential distribution with rate parameter $q\geq0$ (for $q=0$, we have $\boldsymbol{e}_0\equiv\infty$) independent of $(X_t)_{t\geq0}$. For any integer $n\geq1$, we have\begin{align}
    \mathbb{E}\left[\left(I_{\boldsymbol{e}_q}\right)^n\right]=\frac{n!}{(q+\Psi(\chi))\cdots(q+\Psi(\chi^n))}.
\end{align}
\end{proposition}
\begin{proof}
For any real $t>0$, by (\ref{eqn:FDD}), we have \begin{align*}
    \mathbb{E}&\left[\left(\int_0^t\chi(X_t)\text{d}t\right)^n\right]\\
    &=n!\cdot\mathbb{E}\left[\int_0^\infty \text{d}t_1\cdots \int_0^\infty \text{d}t_n\mathbf{1}_{\{0<t_1<\cdots<t_n<t\}}\chi(X_{t_1})\cdots\chi(X_{t_n})\right]\\
    &=n!\cdot\int_0^\infty \text{d}t_1\cdots \int_0^\infty \text{d}t_n\mathbf{1}_{\{0<t_1<\cdots<t_n<t\}}e^{-t_1\Psi(\chi^n)}e^{-(t_2-t_1)\Psi(\chi^{n-1})}\cdots e^{-(t_n-t_{n-1})\Psi(\chi)}\\
    &=n!\cdot\int_0^\infty \text{d}u_1\cdots \int_0^\infty \text{d}u_n\mathbf{1}_{\{u_1+\cdots u_n<t\}}e^{-u_1\Psi(\chi^n)}e^{-u_2\Psi(\chi^{n-1})}\cdots e^{-u_n\Psi(\chi)}\\
    &=\left(\prod_{k=1}^n\frac{k}{\Psi(\chi^k)}\right)\mathbb{P}[X_1+\cdots +X_n<t],
\end{align*}
where $X_1,\cdots,X_n$ are independent exponential random variable of parameter $\Psi(\chi),\cdots,\Psi(\chi^n)$ respectively. Therefore
\begin{align*}
    \mathbb{E}\left[\left(I_{\boldsymbol{e}_q}\right)^n\right]&=\int_0^\infty\mathbb{E}\left[\left(I_t\right)^n\right]qe^{-qt}\text{d}t\\
    &=\left(\prod_{k=1}^n\frac{k}{\Psi(\chi^k)}\right)\int_0^\infty \mathbb{P}[X_1+\cdots+X_n<t]qe^{-qt}\text{d}t\\
&=\left(\prod_{k=1}^n\frac{k}{\Psi(\chi^k)}\right)\prod_{k=1}^n \mathbb{E}\left[e^{-qX_k}\right],
\end{align*}
whence the result follows easily.
\end{proof}

\begin{remark}
Let us notice that if $\mathbb{M}$ is an idempotent Feller topological monoid, then $I_{\boldsymbol{e}_q}$ is an exponential random variable of parameter $q+\Psi(\chi)$.
\end{remark}

\subsubsection{Bochner's Subordination}
A key technique in the study of Markov process is the use of Bochner's subordination (from where subordinators take their name), that it is a random time change of a Markov processes $(M_t)_{t\geq0}$ by an independent additive subordinator $(\sigma_t)_{t\geq0}$, which produces a new Markov process $(M_{\sigma_t})_{t\geq0}$. For a detailed exposition we refer the reader to \cite[Section 8.4]{MR1746300} and references therein. The following result can be seen as an extension of Proposition 8.6 from \cite{MR1746300}.

\begin{proposition}
Let $(X_t)_{t\geq0}$ be a subordinator with Laplace exponent $\Psi$. Suppose that  $(\sigma_t)_{t\geq0}$ is an additive subordinator with Laplace exponent $\Phi$ independent of $(X_t)_{t\geq0}$, which could be defined on an enlargement of $(\Omega,\mathscr{F},\mathbb{P})$. Define the stochastic process $(Y_t)_{t\geq0}$ by \begin{align*}
    Y_t=X_{\sigma_t},\quad t\geq0.
\end{align*}
Then $(Y_t)_{t\geq0}$ is a subordinator with Laplace exponent $\Phi\circ\Psi$.
\end{proposition}
\begin{proof}
The fact that $(Y_t)_{t\geq0}$ is a subordinator is clear from the definition. On the other hand, let us observe that $\mathscr{L}_{Y_t}(\chi)=\mathbb{E}[e^{-\sigma_t\cdot\Psi(\chi)}]=e^{-t\cdot\Phi(\Psi(\chi))},\;t\geq0$. Therefore, the Laplace exponent of $(Y_t)_{t\geq0}$ is given by $\Phi\circ\Psi$. 
\end{proof}

\section{Feller Semigroups on Feller Topological Monoids}\label{sec:FellSemOnTopFellMon}

\subsection{Feller Semigroups}
Our main references are Chapter 17 of Kallenberg's book \cite{MR4226142} and Chapter III of Revuz and Yor's book \cite{MR1725357}. Throughout this subsection, $E$ denotes a locally compact second countable Hausdorff space, and $C_0(E)$ is the Banach space of continuous functions $f:E\to\mathbb{R}$ vanishing at infinity, i.e. $\lim_{x\to\infty}f(x)=0$, equipped with the supremum norm $||f||=\sup_{x\in E}|f(x)|$.  A family $(T_t)_{t\geq0}$ of bounded linear operators on $C_0(E)$ is called a \textit{Feller semigroup} on $C_0(E)$ if $T_0=\text{Id}$, $||T_t||\leq1$ for every $t\geq0$, $T_{s+t}=T_sT_t$ for every $s,t\geq0$ and $\lim_{t\downarrow0}||T_tf-f||=0$ for all $f\in C_0(E)$. For any $\lambda>0$, the resolvent operator $R_\lambda$ is defined by \begin{align*}
    R_\lambda f=\int_{[0,\infty)}e^{-\lambda t}(T_tf)\,\text{d}t,\quad f\in C_0(E).
\end{align*} 
Now we recall the so-called \textit{Yosida approximation} which will be useful later on.
\begin{lemma}[Yosida Approximation]\label{YA} Let $(T_t)_{t\geq0}$ be a Feller semigroup defined on $C_0(E)$. Then \begin{align}
\lim_{\lambda\to\infty}T_t^{(\lambda)}f=T_tf,\quad f\in C_0(E),
\end{align}
where $\Big(T_t^{(\lambda)}\Big)_{t\geq0}$ is defined as $
    T_t^{(\lambda)}=e^{\lambda t(\lambda R_\lambda-I)}$.
\end{lemma}
For the proof of this result the reader can refer to the book of Kallenberg \cite[Lemma 17.7]{MR4226142}.
\subsection{Convolution Semigroups}
In order to study the Laplace transform of subordinators from a Feller semigroup approach, we briefly introduce some notions which are analogous to the classical ones existing for classical subordinators.

\begin{definition}[$\oplus$-Convolution]
Let $\mu$ and $\nu$ be two probability measures on $(\mathbb{M},\mathcal{B}(\mathbb{M}))$. The  $\oplus$-convolution of $\mu$ and $\nu$, denoted by $\mu \boldsymbol{\oplus}\nu$, is the probability measure on $(\mathbb{M},\mathcal{B}(\mathbb{M}))$ such that \begin{align*}
\int_{\mathbb{M}}f(z)\, \big(\mu \boldsymbol{\oplus}\nu\big)(dz)=\int_{\mathbb{M}\times\mathbb{M}} f(x\oplus y)\, \mu(dx)\nu(dy), \quad f\in C_0(\mathbb{M}).
\end{align*} 
\end{definition}
\begin{remark} Observe that $f\mapsto \int_{\mathbb{R}_+\times\mathbb{R}_+} f(x\oplus y)\, \mu(dx)\nu(dy)$ is a positive continuous linear functional on $C_0(\mathbb{M})$, so that the Riesz-Markov theorem guarantees the existence and uniqueness of the probability measure $\mu\boldsymbol{\oplus}\nu$.
\end{remark}

\begin{lemma}[Characterization of  the  $\oplus$-Convolution]\label{CEC}
For $\mu,\nu$ and $\lambda$ probability measures on $(\mathbb{M},\mathcal{B}(\mathbb{M}))$,  $\mu \boldsymbol{\oplus}\nu=\lambda$ if and only if \begin{align}\label{mvl}
    \int_\mathbb{M}\chi(x)\mu(\text{d}x)\cdot \int_\mathbb{M}\chi(x)\nu(\text{d}x) =\int_\mathbb{M}\chi(x)\lambda(\text{d}x),\quad \chi\in\widetilde{\mathbb{M}}.
\end{align}
\end{lemma}
\begin{proof}
Consider $\mathcal{H}:=\{f\in\mathcal{M}_b(\mathbb{M}): \int_{\mathbb{M}}f(z)\, \lambda(dz)=\int_{\mathbb{M}\times\mathbb{M}} f(x\oplus y)\, \mu(dx)\nu(dy)\}$, where $\mathcal{M}_b(\mathbb{M})$ denotes the space of bounded $\mathcal{B}(\mathbb{M})$-measurable functions from $\mathbb{M}$ into $\mathbb{R}$.
 
For sufficiency, suppose that $\mu \boldsymbol{\oplus}\nu=\lambda$. This implies that $C_0(\mathbb{M})\subseteq \mathcal{H}$. Since $C_0(\mathbb{M})$ is closed under pointwise multiplication then, by the functional version of the Monotone Class Theorem \ref{MCT}, $\mathcal{H}$ contains all the bounded $\sigma\big( C_0(\mathbb{M}) \big)$-measurable functions. Remember  that the Baire  $\sigma$-algebra $\sigma\big( C_0(\mathbb{M}) \big)$ coincides with the Borel sigma algebra $\mathcal{B}(\mathbb{M})$, this is due to the fact that $\mathbb{M}$ is locally compact, separable and metrizable. So, in particular, $\mathcal{H}$ contains the functions $\chi\in\widetilde{\mathbb{M}}$ and this clearly implies (\ref{mvl}).
 
To prove necessity, suppose that (\ref{mvl}) holds, then it follows that $\widetilde{\mathbb{M}}\subseteq\mathcal{H}$. By the same argument as before, we get that $\mathcal{H}$ contains all the bounded $\sigma(\widetilde{\mathbb{M}})$-measurable functions. By $(i)$ from Proposition \ref{pro:SigMet}, $\sigma(\widetilde{\mathbb{M}})=\mathcal{B}(\mathbb{M})$. In particular, $\mathcal{H}$ contains $C_0(\mathbb{M})$ so that $\mu \boldsymbol{\oplus}\nu=\lambda$.
\end{proof}

\begin{definition}[Convolution Semigroup of Measures]\label{CSM} An $\oplus$-{convolution semigroup of measures} is a family $(\mu_t)_{t\geq0}$ of probability measures on $(\mathbb{M},\mathcal{B}(\mathbb{M}))$ such that:
\begin{enumerate}
\item $\mu_0=\delta_{\boldsymbol{e}}$;
\item $\mu_{s}\boldsymbol{\oplus}\mu_t=\mu_{s+t}$;
\item $\mu_t\to\delta_{e}$, as $t\downarrow0$, vaguely.
\end{enumerate}
\end{definition}

\begin{remark}
Let $(X_t)_{t\geq0}$ be a subordinator. For every $t\geq0$, let $\mu_t$ be distribution of $X_t$. Then the family $(\mu_t)_{t\geq0}$ is a convolution semigroup of measures.
\end{remark}

The proof of the Lévy-Khintchine representation (\ref{eqn:Levy-Khin}) in the following theorem is based on the Yosida approximation just as the proof of Theorem 3.4.12 in \cite{MR3967735}. According to Applebaum (see the remark in \cite[p. 64]{MR3967735}) this approach is due to Zabczyk.

\begin{theorem}\label{thm:LeKhiRep}
Let $(\mu_t)_{t\geq0}$ be a family of probability measures on $(\mathbb{M},\mathcal{B}(\mathbb{M}))$. The following conditions are equivalent:

\begin{enumerate}[label=(\alph*)]
\item  The transition probabilities defined by  \begin{align}\label{PtxA}
P_{t}(x, A):=P_t 1_A(x)=\int_{\mathbb{M}} {1}_{A}(x\oplus y)\, \mu_{t}(d y),\quad A\in \mathcal{B}(\mathbb{M}), \, t\geq0,
\end{align}determine a Feller semigroup on $C_0(\mathbb{M})$, that is, the familiy $(P_t)_{t\geq0}$ of operators on $\mathcal{M}_b(\mathbb{M})$ defined by  $P_tf(x):=\int_{\mathbb{M}}f(x)P_t(x,\text{d}y),\, x\in\mathbb{M}$, is a Feller semigroup on $C_0(\mathbb{M})$.
\item The family $(\mu_t)_{t\geq0}$ is an $\oplus$-convolution semigroup of measures.
\end{enumerate}
Moreover, suppose that conditions \emph{\textbf{(I)}}, $(a)$ and/or $(b)$ are satisfied, then the following Lévy-Khintchine representation holds 
\begin{align}\label{eqn:Levy-Khin}
    \mathscr{L}_{\mu_t}(\chi)=\exp\left\{-t\left(\alpha(\chi)-\int_{\mathbb{M}\backslash\{ \boldsymbol{e}\}}\big(1-\chi(x)\big)\Pi(\text{d}x)\right)\right\},\quad \chi\in\widetilde{\mathbb{M}},\;t\geq0.
\end{align}
for some measure $\Pi$ on $(\mathbb{M}\backslash\{ \boldsymbol{e}\},\mathcal{B}(\mathbb{M\backslash\{ \boldsymbol{e}\}}))$.
\end{theorem}
\begin{proof}
We begin the proof with some preliminary observations. An application of the dominated convergence theorem guarantees that $P_t\big( C_b(\mathbb{M})\big)\subseteq C_b(\mathbb{M})$ for all $t\geq0$, where $C_b(\mathbb{M})$ denotes the space of bounded continuous functions from $(\mathbb{M},\tau)$ into $\mathbb{R}$. Indeed, $x\mapsto x\oplus y$ is a continuous function and if $x_n\to x$ in $\mathbb{M}$ as $n\to\infty$, then \begin{align*}
\lim\limits_{n\to\infty}P_tf(x_n)=\int_{\mathbb{M}}\lim\limits_{n\to\infty}f(x_n\oplus y)\,\mu_t({d}y)=\int_{\mathbb{M}}f(x\oplus y)\,\mu_t({d}y)=P_tf(x),
\end{align*}
Now, if $f\in C_{0}(\mathbb{M})$, once again the dominated convergence theorem and part $(iii)$ from Lemma \ref{lem:property} assert \begin{align*}
\lim\limits_{x\to \partial_\infty}P_tf(x)=\int_{\mathbb{M}}\lim\limits_{x\to \partial_\infty}f(x\oplus y)\mu_{t}({d}y)=0,\quad f\in C_0(\mathbb{M}),
\end{align*}
so that $P_t\big( C_0(\mathbb{M})\big)\subseteq C_0(\mathbb{M})$, for all $t\geq0$.

Let us start by verifying the equivalence of conditions $(a)$ and $(b)$.  Given $f\in C_0(\mathbb{M})$ and $x\in\mathbb{M}$, denote as $f_x$ the mapping from $\mathbb{M}$ into $\mathbb{R}$ defined  by $f_x(y):=f(x\oplus y)$, $y\in\mathbb{M}$. We have further $f_x\in C_0(\mathbb{M})$. We can now observe that
 \begin{align}\label{Ps+t}
P_{s+t}f(x)=\int_{\mathbb{M}}f(x\oplus y)\mu_{s+t}(dy)=\int_{\mathbb{M}}f_x( y)\mu_{s+t}(dy),\quad s,t\geq0.
\end{align}
Moreover, for every $s,t\geq0$,
\begin{align}\label{PsPt}
\begin{split}
P_tP_sf(x)&=\int_{\mathbb{M}}P_sf(x\oplus z)\,\mu_t(\text{d}z)\\
&=\int_{\mathbb{M}\times \mathbb{M}} f (y\oplus z\oplus x)\mu_{t}(dz)\mu_s(dy)\\
&=\int_{\mathbb{M}\times \mathbb{M}} f_x (y\oplus z)\mu_{t}(dz)\mu_s(dy) .
\end{split}
\end{align}
Then the semigroup property of $(P_t)_{t\geq0}$ is read as
\begin{align}
\int_{\mathbb{M}}f_x( y)\mu_{s+t}(dy)=\int_{\mathbb{M}\times \mathbb{M}} f_x (y\oplus z)\,\mu_{t}(dz)\mu_s(dy),\quad x\in \mathbb{M}, \; f\in C_0(\mathbb{M}),
\end{align}
and this is equivalent to  $\mu_t\boldsymbol{\oplus} \mu_s=\mu_{s+t},\;s,t\geq0$.
Recall that $\lim _{t \downarrow 0}\left\|P_{t} f-f\right\|=0$ for each $ f \in C_{0}(\mathbb{M})$ is equivalent to $\lim _{t \downarrow 0} P_{t} f(x)=f(x),\; x\in\mathbb{M},\, f\in C_0(\mathbb{M})$ (see for example \cite[p. 89]{MR1725357}), but this is condition $2$ of Definition \ref{CSM}. Thus, we have verified that condition $(a)$ holds if and only if condition $(b)$ is true.

To prove the last statement of the result, we need the following lemma. Let us recall that $\widetilde{C}(\mathbb{M})$ is the space of all real-valued continuous functions on $(\mathbb{M},\tau_{\widetilde{\mathbb{M}}})$ (see \eqref{eqn:rhometric}).

\begin{lemma}
If $(P_t)_{t\geq0}$ is a Feller semigroup on $C_0(\mathbb{M})$, then $(P_t)_{t\geq0}$ is a $C_0$-semigroup on $\widetilde{C}(\mathbb{M})$.
\end{lemma}
\begin{proof}
Since for every $t\geq0$ we have $P_t\chi(\boldsymbol{e})=\int_{\mathbb{M}}\chi(e\oplus y)\,\mu_t{(\text{d}y)}=\int_{\mathbb{M}}\chi(y)\,\mu_t{(\text{d}y)}$, $\chi\in\widetilde{\mathbb{M}}$, we obtain
\begin{align}\label{PtChi}
    P_t\chi(x)=\int_{\mathbb{M}}\chi(x\oplus y)\,\mu_t(\text{d}y)=\chi(x)\cdot \int_{\mathbb{M}}\chi(y)\,\mu_t(\text{d}y)=\chi(x)\cdot P_t\chi(\boldsymbol{e}).
\end{align}
On the other hand, consider \begin{align*}
\mathcal{H}_1&:=\{f\in\mathcal{M}_b(\mathbb{M}): P_tP_sf=P_{s+t}f,\text{ for any }s,t\geq0 \}.    
\end{align*} 
If $(P_t)_{t\geq0}$ is a Feller semigroup on $C_0(\mathbb{M})$ then $C_0(\mathbb{M})\subseteq\mathcal{H}_1$. By the same argument as in Lemma \ref{mvl}, we get $\mathcal{H}_1=\mathcal{M}_b(\mathbb{M})$. In particular, by (\ref{PsPt}) and (\ref{PtChi}), we get \begin{align*}
    P_{s+t}\chi(\boldsymbol{e})&=\int_{\mathbb{M}}\chi(x)\,\mu_{s+t}(\text{d}x),\\
    P_sP_t\chi(\boldsymbol{e})&=\int_{\mathbb{M}}\chi(x)\,\mu_t(\text{d}x)\cdot\int_{\mathbb{M}}\chi(x)\,\mu_s(\text{d}x)=P_t\chi(\boldsymbol{e})\cdot P_s\chi(\boldsymbol{e}),
\end{align*}
so that the semigroup property implies that \begin{align}
    P_{s+t}\chi(\boldsymbol{e})=P_t\chi(\boldsymbol{e})\cdot P_s\chi(\boldsymbol{e}),\quad s,t\geq0.
\end{align}
Therefore $t\mapsto P_{t}\chi(\boldsymbol{e})$ is non-increasing, and then \begin{align}\label{PtChiet}
    P_t\chi(\boldsymbol{e})=\Big(P_1\chi(\boldsymbol{e})\Big)^t,\quad t\geq0.
\end{align}
Now, let
\begin{align*}
    \mathcal{H}_2&:=\{f\in\mathcal{M}_b(\mathbb{M}): P_tf\uparrow f,\text{ as }t\downarrow0 \}.    
\end{align*}
By (\ref{PtChi}) and (\ref{PtChiet}, we get 
\begin{align*}
    P_t\chi(x)=\chi(x)\cdot \left(P_1\chi(\boldsymbol{e})\right)^t.
\end{align*}
Then $\widetilde{\mathbb{M}}\subseteq\mathcal{H}_2$ which implies that $\mathcal{H}_2=\mathcal{M}_b(\mathbb{M})$. In particular, $\widetilde{C}(\mathbb{M})\subseteq\mathcal{H}_2$, that is \begin{align*}
    \lim_{t\downarrow0}P_tf(x)\uparrow f(x),\quad x\in\mathbb{M},\;f\in\widetilde{C}(\mathbb{M}).
\end{align*}
But, by part $(ii)$ in Proposition \ref{pro:SigMet}, the space $(\mathbb{M}^*,\rho_{\widetilde{\mathbb{M}}})$ is compact and therefore, by Dini's Theorem, the preceding convergence is uniform for any $f\in\widetilde{C}(\mathbb{M})$.
\end{proof}

We are now in position to prove the Lévy-Khintchine representation. Let us first observe that, for any $f\in\widetilde{C}(\mathbb{M})$, we have for $\lambda>0$ \begin{align*}
    \lambda R_\lambda f(\boldsymbol{e})&=\lambda\int_{[0,\infty)}e^{-\lambda t}P_tf(\boldsymbol{e})\,\text{d}t\\
    &=\lambda\int_{[0,\infty)}e^{-\lambda t}\int_{\mathbb{M}}f( y)\,\mu_t(\text{d}y)\,\text{d}t\\
    &=\int_{\mathbb{M}}f(  y)\,\pi_\lambda(\text{d}y),
\end{align*} 
where $\pi_{\lambda}$ is the probability measure defined by $\pi_{\lambda}(\text{d}y)=\lambda\int_{[0,+\infty)}e^{-\lambda t}\mu_t(\text{d}y)\,\text{d}t$.

Now let $\widetilde{\varphi}\in\widetilde{C}(\mathbb{M})$ defined by $\widetilde{\varphi}(x)=\sum_{n\geq1}\chi_n(x)/2^n,$ $x\in\mathbb{M}$. Since $(P_t)_{t\geq0}$ is a $C_0$-semigroup on $\widetilde{C}(\mathbb{M})$, by the Yosida Approximation (Lemma \ref{YA}) we have \begin{align*}
   \int_{\mathbb{M}}\widetilde{\varphi}(x)\mu_t(\text{d}x)&= P_t\widetilde{\varphi}(\boldsymbol{e})\\
    &=\lim_{\lambda\to\infty}P_t^{(\lambda)}\widetilde{\varphi}(\boldsymbol{e})\\
    &=\lim_{\lambda\to\infty}e^{\lambda(\lambda R_\lambda-I)}\widetilde{\varphi}(\boldsymbol{e})\\
    &=\lim_{\lambda\to\infty}\exp\left\{\lambda\left(\lambda R_\lambda\widetilde{\varphi}(\boldsymbol{e})- \widetilde{\varphi}(\boldsymbol{e})\right)\right\}\\
    &=\lim_{\lambda\to\infty}\exp\left\{\lambda\left(\int_{\mathbb{M}}\widetilde{\varphi}(y)\pi_{\lambda}(\text{d}y)- 1\right)\right\}\\
    &=\lim_{\lambda\to\infty}\exp\left\{-t\left(\int_{\mathbb{M}}\big( 1-\widetilde{\varphi}(y)\big) \, \lambda\pi_{\lambda}(\text{d}y)\right)\right\}\\
    &=\lim_{\lambda\to\infty}\exp\left\{-t\left(\int_{\mathbb{M}}\varphi_{\widetilde{\mathbb{M}}}(y) \, \widetilde{\pi}_\lambda(\text{d}y)\right)\right\}.
\end{align*}
Analogously, we have \begin{align*}
    \int_{\mathbb{M}}\chi(x)\,\mu_{t}(\text{d}x)&=\lim_{\lambda\to\infty}\exp\left\{-t\left(\int_{\mathbb{M}}\big( 1-\chi(y)\big) \, \widetilde{\pi}_{\lambda}(\text{d}y)\right)\right\}\\
    &=\lim_{\lambda\to\infty}\exp\left\{-t\left(\int_{\mathbb{M}}\frac{1-\chi(y)}{\varphi_{\widetilde{\mathbb{M}}}(y)}\, \widetilde{\Pi}_{\lambda}(\text{d}y)\right)\right\}\\
    &=\exp\left\{-t\left(\int_{\mathbb{M}\cup{\partial_\infty}}\alpha(\chi)1_{\{y=\boldsymbol{e}\}}+\frac{1-\chi(y)}{\varphi_{\widetilde{\mathbb{M}}}(y)}\, \widetilde{\Pi}(\text{d}y)\right)\right\}\\
    &=\exp\left\{-t\alpha(\chi)/ \widetilde{\Pi}(\{\boldsymbol{e}\})-t \left(\int_{\mathbb{M}\backslash\{\boldsymbol{e}\}}\big(1-\chi(y)\big)\, {\Pi}(\text{d}y)\right)-t\cdot\Pi(\{\partial_\infty\})\right\},
\end{align*}
where $\widetilde{\Pi}_{\lambda}(\text{d}y)=\varphi_{\widetilde{\mathbb{M}}}(y)\widetilde{\pi}_{\lambda}(\text{d}y)$.
\end{proof}



\subsection{Invariance Principle}
In this subsection, we suppose that the Feller topological monoid $\mathbb{M}$ possesses a left semigroup action, namely $
    \cdot:(0,+\infty)\times\mathbb{M}\to\mathbb{M}$,
which satisfies $r\cdot (x\oplus y)=rx\oplus rx$ for all $r\in(0,+\infty)$, $x,y\in\mathbb{M}$.
\begin{theorem}
Let $\left\{\zeta_{n}, n \geq 1\right\}$ be a sequence of i.i.d. random $\mathbb{M}$-valued elements such that \begin{align*}
b_{n}^{-1}\left(\zeta_{1}\oplus\cdots\oplus\zeta_{n}\right) \Rightarrow \xi,
\end{align*}
where $\Rightarrow$ denotes the weak convergence on $(\mathbb{M},\tau_{\widetilde{\mathbb{M}}})$, and ${b_n,n\geq1}$ is a sequence of positive constants. Now consider the Markov chains defined by \begin{align*}
\left(X_k^{(n)}(x):=\frac{( b_n\cdot x)\oplus\left(\bigoplus_{i=1}^k\zeta_i \right)}{b_n}\right)_{k\geq0},\quad n\geq0.
\end{align*}
Then, the continuous time stochastic processes $\{(X_{t}^{(n)},t\geq0):=(X^{(n)}_{\lfloor n t \rfloor}(\boldsymbol{e}),t\geq0),\, n\geq1\}$ converge on the $J_1$ Skorokhod space $\mathcal{D}(\mathbb{R}_+,(\mathbb{M},\rho_{\widetilde{\mathbb{M}}}))$ to the subordinator $(X_t)_{t\geq0}$ such that $X_1\overset{d}{=}\xi$.
\end{theorem}
\begin{proof}
Denote as $\mu_t^{(n)}$ and $\mu_t$ the measures induced by the random elements $b_{n}^{-1}\left(\zeta_{1}\oplus\cdots\oplus\zeta_{{\lfloor nt \rfloor }}\right)$ and $X_t$ respectively. The Feller semigroups of $(X_t^{(n)},t\geq0)$ and $(X_t)_{t\geq0}$ are given by $P_t^{(n)}f(x)=\int_{\mathbb{M}}f(x\oplus y)\mu_t^{(n)}(\text{d}y)$ and $P_tf(x)=\int_{\mathbb{M}}f(x\oplus y)\mu_t(\text{d}y)$ respectively. By Theorem 17.28 in \cite{MR4226142} it is enough to prove $
    \lim_{n\to\infty}\sup_{x\in\mathbb{M}}\left|P_t^{n}f(x)-P_tf(x)\right|=0$ for each $f\in\widetilde{C}_0(\mathbb{M})$, where $\widetilde{C}_0(\mathbb{M})$ is the class of real-valued continuous functions on $(\mathbb{M},\rho_{\widetilde{\mathbb{M}}})$ vanishing at infinity.

Observe that if $b_{n}^{-1}\left(\zeta_{1}\oplus\cdots\oplus\zeta_{n}\right) \Rightarrow \xi$, where $\xi$ is such that $\mathscr{L}_\xi(\chi)=\exp\{-\Psi(\chi)\}$, then
\begin{align*}
    \lim_{n\to\infty}\left(\mathscr{L}_{ b_n^{-1}\cdot\zeta_1}(\chi)\right)^n=e^{-\Psi(\chi)}, \quad\chi\in\widetilde{\mathbb{M}},
\end{align*}
since $\widetilde{\mathbb{M}}\subseteq\widetilde{C}(\mathbb{M})$. Therefore, for every $t\geq0$,
\begin{align*}
    \lim_{n\to\infty}\left(\mathscr{L}_{ b_n^{-1}\cdot\zeta_1}(\chi)\right)^{\lfloor  nt\rfloor}=e^{-t\Psi(\chi)},\quad \chi\in\widetilde{\mathbb{M}},
\end{align*}
which, by Lemma \ref{lem:weakconv}, implies that
\begin{align*}
    b_{n}^{-1}\left(\zeta_{1}\oplus\cdots\oplus\zeta_{{\lfloor nt \rfloor }}\right) \Rightarrow X_t
\end{align*}
on $(\mathbb{M},\tau_{\widetilde{\mathbb{M}}})$ since  $\mathscr{L}_{X_t}(\chi)=\exp\{-t\Psi(\chi)\}$. Then, it follows that \begin{align*}
    \lim\limits_{n\to\infty}\sup_{x\in\mathbb{M}}\left|\int_{\mathbb{M}}\chi(x\oplus y)\mu_t^{(n)}(\text{d}y)-\int_{\mathbb{M}}\chi(x\oplus y)\mu_t(\text{d}y)\right|=0,\quad \chi\in\widetilde{\mathbb{M}}.
\end{align*}
By the Stone–Weierstrass theorem, we finally obtain\begin{align*}
    \lim\limits_{n\to\infty}\sup_{x\in\mathbb{M}}\left|\int_{\mathbb{M}}f(x\oplus y)\mu_t^{(n)}(\text{d}y)- \int_{\mathbb{M}}f(x\oplus y)\mu_t(\text{d}y)\right|=0,\quad f\in \widetilde{C}_0(\mathbb{M}).
\end{align*}
\end{proof}

\section{Lévy-Itô Decomposition of Subordinators on Feller Topological Monoids}\label{sec:LevyItoDec}
Let $(\mathbb{M},\widetilde{\mathbb{M}})$ be a Feller topological monoid. In this last section we obtain a general Lévy-Itô decomposition for subordinators on $(\mathbb{M},\widetilde{\mathbb{M}})$. This decomposition unifies and extends previous results (see Table \ref{tab:LevyKhintandItoInLit}). We will denote as $\text{d}t$ the Lebesgue measure on $(0,\infty)$ and by $\mu\otimes\nu$ the product measure of two measures $\mu$ and $\nu$.

\begin{theorem}[Lévy-Itô Decomposition] Let $\mathcal{N}=\sum_i \delta_{(t_i,x_i)}$ be a Poisson point process on $(0,\infty)\times \mathbb{M}$ with intensity measure $\text{d}t\otimes\Pi$ where $\Pi$ is a $\sigma$-finite measure. Then, for every $t\geq0$,
\begin{align}\label{LevyItoDec}
\mathbb{E}\left[\chi\left(\bigoplus_{t_i\leq t} x_i\right)\right]=\exp\left\{-t\int_{\mathbb{M}}\big(1-\chi(x)\big)\,\Pi(\text{d}x)\right\},\qquad \chi\in\widetilde{\mathbb{M}}.
\end{align}
\end{theorem}
\begin{proof}
    Let us first prove that (\ref{LevyItoDec}) holds when $\Pi$ is a finite measure.  Define $\Psi(\chi)=\int_{\mathbb{M}}\big(1-\chi(x)\big)\,\Pi(\text{d}x)$. Since $\Pi$ is a finite measure on $\mathbb{M}$, we can write $\mathcal{N}=\sum_{i=1}^\infty\delta_{(t_i,x_i)}$.  Now we consider, for every $t>0$, the Laplace transform of $\bigoplus_{t_i\leq t,\;i=1,\cdots, n}x_i$. By the memoryless property of marked Poisson process \cite[Theorem 7.4]{MR3791470}, we get
\begin{align*}
    &\mathbb{E}\left[\chi\left(\bigoplus_{t_i\leq t, \,i=1,\cdots, n} x_i\right)\right]\\
    &=\mathbb{E}\left[\prod_{i=1}^n \Big(1-\big(1-\chi(x_i)\big)1_{\{t_i\leq t\}}\Big)\right]\\
    &=\int_{((0,\infty)\times\mathbb{M})^n}  \mathbf{1}_{\{t_1<\cdots<t_n\}}\prod_{i=1}^n \Big(1-\big(1-\chi(y_i)\big)1_{\{t_i\leq t\}}\Big) e^{-t_n\Pi(\mathbb{M})}\text{d}t_1\Pi(\text{d}y_1)\cdots\text{d}t_n\Pi(\text{d}y_n)\\
    &=\sum_{i=0}^{n-1}\int_{((0,\infty)\times\mathbb{M})^n} \mathbf{1}_{\{t_1<\cdots<t_{i}<t<t_{i+1}<\cdots<t_n\} }\prod_{j=1}^i\big( 1-(1-\chi(y_j))\big)e^{-t_n\Pi(\mathbb{M})}\text{d}t_1\Pi(\text{d}y_1)\cdots\text{d}t_n\Pi(\text{d}y_n)\\
    &\quad+\int_{((0,\infty)\times\mathbb{M})^n}  \mathbf{1}_{\{t_1<\cdots<t_n<t\}}\prod_{i=1}^{n} \Big(1-\big(1-\chi(y_i)\big)\Big) e^{-t_n\Pi(\mathbb{M})}\text{d}t_1\Pi(\text{d}y_1)\cdots\text{d}t_n\Pi(\text{d}y_n)\\
    &= \sum_{i=0}^{n-1}\frac{\big( \Pi(\mathbb{M})-\Psi(\chi) \big)^i}{\Pi(\mathbb{M})^{i-n}}\int_{(0,+\infty)^n} \mathbf{1}_{\{t_1<\cdots<t_{i}<t<t_{i+1}<\cdots<t_n\} }e^{-t_n\Pi(\mathbb{M})}\text{d}t_1\cdots\text{d}t_n\\
    &\quad+{\big(\Pi(\mathbb{M})-\Psi(\chi)\big)^n} \int_{(0,+\infty)^n} \mathbf{1}_{\{t_1<\cdots<t_n<t\} }e^{-t_n\Pi(\mathbb{M})}\text{d}t_1\cdots\text{d}t_n.
    \end{align*}
    The first term in the above line equals 
    \begin{align*}
    &\sum_{i=0}^{n-1}\frac{\big( \Pi(\mathbb{M})-\Psi(\chi) \big)^i}{\Pi(\mathbb{M})^{i-n}}\int_0^{t}\text{d}t_i\cdots \int_{0}^{t_2}\text{d}t_1\int_{t}^\infty\text{d}t_{i+1}\int_{t_{i+1}}^\infty\text{d}t_{i+2}\cdots\int_{t_{n-2}}^\infty \text{d}t_{n-1} \int_{t_{n-1}}^\infty\text{d}t_ne^{-t_n\Pi(\mathbb{M})}\\
    &=\sum_{i=0}^{n-1}\big(\Pi(\mathbb{M})-\Psi(\chi)\big)^{i}\int_0^{t}\text{d}t_i\cdots \int_{0}^{t_2}\text{d}t_1\,  e^{-t\Pi(\mathbb{M})}\\
    &=e^{-t\Pi(\mathbb{M})}\sum_{i=0}^{n-1} \frac{\Big(t\big(\Pi(\mathbb{M})-\Psi(\chi)\big)\Big)^i}{i!}.
\end{align*}
Therefore \begin{align}\label{LevyItoDecDelta}
    \mathbb{E}\left[\chi\left(\bigoplus_{t_i\leq t, \,i=1,\cdots, n} x_i\right)\right]=e^{-t\Pi(\mathbb{M})}\sum_{i=0}^{n-1} \frac{\Big(t\big(\Pi(\mathbb{M})-\Psi(\chi)\big)\Big)^i}{i!}+\Delta_n,\quad n\geq1,
\end{align}
where $\Delta_n:={\big(\Pi(\mathbb{M})-\Psi(\chi)\big)^n} \int 1_{\{t_1<\cdots<t_n<t\} }e^{-t_n\Pi(\mathbb{M})}\text{d}t_1\cdots\text{d}t_n$. Since $\Pi$ is a finite measure on $\mathbb{M}$, it follows that the (random) set $\{x_i\in\mathbb{M}: t_i\leq t\}$ is a.s. finite for every $t>0$, and therefore
\begin{align}\label{Limchin}
    \lim_{n\to\infty} \mathbb{E}\left[\chi\left(\bigoplus_{t_i\leq t, \,i=1,\cdots, n} x_i\right)\right]=\mathbb{E}\left[\chi\left(\bigoplus_{t_i\leq t} x_i\right)\right].
\end{align}
Letting $n$ tend to infinity in (\ref{LevyItoDecDelta}) and using, on one hand that $|\Delta_n|\leq |\Pi(\mathbb{M})-\Psi(\chi)\big|^n t^n/n!$ for every $n\geq1$, on the other hand (\ref{Limchin}), we get
\begin{align*}
\mathbb{E}\left[\chi\left(\bigoplus_{t_i\leq t} x_i\right)\right]=\exp\{-t\Psi(\chi)\}=\exp\left\{-t\int_{\mathbb{M}}\big(1-\chi(x)\big)\,\Pi(\text{d}x)\right\},\qquad \chi\in\widetilde{\mathbb{M}},
\end{align*}
for $\Pi$ a finite measure on $\mathbb{M}$.

Let now $\Pi$ be a $\sigma$-finite measure on $\mathbb{M}$. We can then choose an increasing sequence $(K_n)_{n\geq1}$ of Borel sets of $\mathbb{M}$ such that $\cup_{n\geq1}K_n=\mathbb{M}$ and $\Pi(K_n)<\infty$ for $n\geq1$. Denote as $\mathcal{N}^{(n)}$ the Poisson point process    $\mathcal{N}$ restricted  to $(0,\infty)\times K_n\subseteq(0,\infty)\times\mathbb{M}$. By applying the result just proved to $\mathcal{N}^{(n)}$, we get
\begin{align}\label{ProtoLevyKhi}
\mathbb{E}\left[\chi\left(\bigoplus_{t_i\leq t} x_i^{(n)}\right)\right]=\exp\left\{-t\int_{K_n}\big(1-\chi(x)\big)\,\Pi(\text{d}x)\right\},\quad n\geq1,
\end{align}
where $\left\{x_i^{(n)}\right\}$ are the marks of $\mathcal{N}=\sum_i\delta_{(t_i,x_i)}$ that belong to $K_n$. On the other hand, since $\cup_{n\geq1}K_n=\mathbb{M}$, we have $
    \bigoplus_{t_i\leq t} x_i=\lim_{n\to\infty}\bigoplus_{t_i\leq t} x_i^{(n)}$
and moreover, $\lim_{n\to\infty}\chi\left(\bigoplus_{t_i\leq t} x_i^{(n)}\right)=\chi\left(\bigoplus_{t_i\leq t}x_i\right)$. Finally, letting $n$ tend to infinity in (\ref{ProtoLevyKhi}), we obtain the desired result.
\end{proof}

\begin{theorem}
Consider $\mathcal{N}=\sum_i \delta_{(t_i,x_i)}$ a Poisson point process on $(0,\infty)\times \mathbb{M}$ with intensity $\text{d}t\times\Pi$ with $\Pi$ a $\sigma$-finite measure such that $\int_{\mathbb{M}}\big(1-\chi(x)\big)\Pi(\text{d}x)<+\infty$ for every $\chi$. Then, the process defined by $X_t:=\bigoplus_{t_i\leq t}x_i$ is a subordinator in the Feller topological monoid $\mathbb{M}$.
\end{theorem}
\begin{proof}
Let us observe that, for $s,t\geq0$,
\begin{align*}
X_{s+t}=\left(\bigoplus_{t_i\leq t}x_i\right)\oplus\left(\bigoplus_{t<t_i\leq s+t}x_i\right)=:X_t\oplus X_s^{(t)},
\end{align*}
 $X_t$ is independent of $X^{(t)}_s$, and $X_s\overset{d}{=}X_s^{(t)}$. Recall that the process $Y_t:=\chi(X_t)\exp\{t\Psi(\chi)\},\; t\geq0$, is a martingale for every $\chi\in\widetilde{\mathbb{M}}$. By the regularization theorem, for every $\chi\in\widetilde{\mathbb{M}}$, $t\mapsto\chi(X_t)$ has a càdlàg modification. Then $\lim_{t\downarrow t_0}\chi(X_t)=\chi(X_{t_0})$ for every $\chi\in\widetilde{\mathbb{M}}$. This implies $\lim_{t\downarrow t_0} X_{t}=X_{t_0}$ since $\mathbb{M}$ is a Feller topological monoid.
\end{proof}

\begin{appendices}
\section{Functional Monotones Class Theorem}
\begin{theorem}[Monotone Class Theorem]\label{MCT}
Let $\mathscr{H}$ be a vector space of bounded real-valued functions such that
\begin{enumerate}
    \item[i)] the constant functions are in $\mathscr{H}$,
    \item[ii)] if $\{h_n\}$ is an increasing sequence of positive elements of $\mathscr{H}$ such that $h=\sup_{n}h_n$ is bounded, then $h\in\mathscr{H}$.
\end{enumerate}
If $\mathscr{C}$ is a subset of $\mathscr{H}$ which is stable under pointwise multiplication, then $\mathscr{H}$ contains all the bounded $\sigma(\mathscr{C})$-mesurable functions.
\end{theorem}
\end{appendices}

\bibliographystyle{amsplain}
\bibliography{references} 

\noindent (UPC) Instituto de Matemáticas, Universidad Nacional Autónoma de México (UNAM), México. \\  Email:
\texttt{ulisesperez@ciencias.unam.mx}

\noindent (GPS) Centro de Investigación en Matemáticas (CIMAT), México.\\ 
Email: \texttt{gerardo.perez@cimat.mx}

\end{document}